\documentclass{article}%
\usepackage{amssymb}
\usepackage{amsfonts}
\usepackage{amsmath}
\usepackage[colorlinks=true, pdfstartview=FitV, linkcolor=blue,
citecolor=blue, urlcolor=blue]{hyperref}
\usepackage{color}
\usepackage{graphicx}
\usepackage[compress]{cite}%

\newtheorem{theorem}{Theorem}

\newtheorem{corollary}[theorem]{Corollary}

\newtheorem{definition}[theorem]{Definition}

\newtheorem{lemma}[theorem]{Lemma}

\newtheorem{remark}[theorem]{Remark}

\newenvironment{proof}[1][Proof]{\noindent\textbf{#1.} }{\ \rule{0.5em}{0.5em}}
\numberwithin{equation}{section}
\numberwithin{theorem}{section}
\allowdisplaybreaks
\begin{document}

\title{Transformation formulas of a character analogue of $\log\theta_{2}\left(
z\right)  $}
\author{Merve \c{C}elebi Bozta\c{s}\thanks{\textit{mcelebi89@yandex.com.tr}} \ and
M\"{u}m\"{u}n Can\thanks{\textit{mcan@akdeniz.edu.tr}}\\\textit{Department of Mathematics, Akdeniz University,}\\\textit{Antalya, 07058, Turkey} }
\date{}
\maketitle

\begin{abstract}
In this paper, transformation formulas for the function
\[
A_{1}\left(  z,s:\chi\right)  =\sum\limits_{n=1}^{\infty}\sum\limits_{m=1}%
^{\infty}\chi\left(  n\right)  \chi\left(  m\right)  \left(  -1\right)
^{m}n^{s-1}e^{2\pi imnz/k}%
\]
are obtained. Sums that appear in transformation formulas are generalizations
of the Hardy--Berndt sums $s_{j}(d,c),$ $j=1,2,5$. As applications of these
transformation formulas, reciprocity formulas for these sums are derived and
several series relations are presented.



\end{abstract}

\section{Introduction}

Hardy sums or Berndt's arithmetic sums are defined for $c>0$ by
\[%
\begin{tabular}
[c]{ll}%
$S(d,c)={\sum\limits_{n=1}^{c-1}}\left(  -1\right)  ^{n+1+\left[  dn/c\right]
},$ & $s_{1}(d,c)={\sum\limits_{n=1}^{c-1}}\left(  -1\right)  ^{\left[
dn/c\right]  }\mathfrak{B}_{1}\left(  \dfrac{n}{c}\right)  ,$ \vspace{0.2cm}\\
$s_{2}(d,c)={\sum\limits_{n=1}^{c-1}}\left(  -1\right)  ^{n}\mathfrak{B}%
_{1}\left(  \dfrac{n}{c}\right)  \mathfrak{B}_{1}\left(  \dfrac{dn}{c}\right)
,$ & $s_{3}(d,c)={\sum\limits_{n=1}^{c-1}}\left(  -1\right)  ^{n}%
\mathfrak{B}_{1}\left(  \dfrac{dn}{c}\right)  ,$ \vspace{0.2cm}\\
$s_{4}(d,c)={\sum\limits_{n=1}^{c-1}}\left(  -1\right)  ^{\left[  dn/c\right]
},$ & $s_{5}(d,c)={\sum\limits_{n=1}^{c-1}}\left(  -1\right)  ^{n+\left[
dn/c\right]  }\mathfrak{B}_{1}\left(  \dfrac{n}{c}\right)  ,$%
\end{tabular}
\ \ \ \
\]
where $\mathfrak{B}_{p}\left(  x\right)  $ are the Bernoulli functions (see
Section 2) and $[x]$ denotes the greatest integer not exceeding $x$. These
sums arise in transformation formulas for the logarithms of the classical
theta functions \cite{b6,g}. In particular, Hardy--Berndt sums $s_{j}(d,c)$
$\left(  j=1,2,5\right)  $ appear in transformation formulas of $\log
\theta_{2}(z):$\ \

Let $Tz=\left(  az+b\right)  /\left(  cz+d\right)  $\ where $a,$\ $b,$%
\ $c$\ and $d$\ are integers with $ad-bc=1$\ and $c>0.$\ Berndt \cite{b6}
proves that if $d$\ is even, then%
\begin{equation}
\log\theta_{2}\left(  Tz\right)  =\log\theta_{4}\left(  z\right)  +\frac{1}%
{2}\log\left(  cz+d\right)  -\frac{\pi ia}{4c}-\frac{\pi i}{4}+\frac{\pi i}%
{2}s_{1}(d,c), \label{6d}%
\end{equation}
if $c$\ is even, then
\begin{equation}
\log\theta_{2}\left(  Tz\right)  =\log\theta_{2}\left(  z\right)  +\frac{1}%
{2}\log\left(  cz+d\right)  +\pi i\frac{a+d}{4c}-\frac{\pi i}{4}-\pi
is_{2}(d,c), \label{6e}%
\end{equation}
and Goldberg \cite{g} shows that if $c$ and $d$\ are odd, then%
\begin{equation}
\log\theta_{2}\left(  Tz\right)  =\log\theta_{3}\left(  z\right)  +\frac{1}%
{2}\log\left(  cz+d\right)  -\frac{\pi ia}{4c}-\frac{\pi i}{4}+\frac{\pi i}%
{2}s_{5}(d,c). \label{6f}%
\end{equation}
Moreover, Goldberg \cite{g} shows that these sums also arise in the theory of
$r_{m}(n)$, the number of representations of $n$ as a sum of $m$ integral
squares and in the study of the Fourier coefficients of the reciprocals of
$\theta_{j}(z),$ $j=2,3,4$. Analogous to Dedekind sums, these sums also
satisfy reciprocity formulas: For coprime positive integers $d$ and $c$ we
have \cite{b6,g}%
\begin{align}
s_{1}(d,c)-2s_{2}(c,d)  &  =\frac{1}{2}-\frac{1}{2}\left(  \frac{1}{dc}%
+\frac{c}{d}\right)  ,\text{ \ if }d\text{ is even,}\label{8}\\
s_{5}(d,c)+s_{5}(c,d)  &  =\frac{1}{2}-\frac{1}{2cd},\text{ \ if }c\text{ and
}d\text{ are odd.} \label{9}%
\end{align}
Various properties of Hardy--Berndt sums have been investigated
(\cite{a3,b6,bg,c1,g,m1,m2,m4,pz,pt,sy,s,wh,zw,zz}) and several
generalizations have been studied (\cite{cck1,ck,cd,dc,dc1,lz,lg,m3}).

A character analogue of classical Dedekind sum, called as Dedekind character
sum, appears in the transformation formula of\ a generalized Eisenstein series
$G\left(  z,s:\chi:r_{1},r_{2}\right)  $ (see (\ref{7a}) below) associated to
a non-principle primitive character $\chi$ of modulus $k$ \cite[p. 12]{b2}.
This sum is defined by%
\[
s\left(  d,c:\chi\right)  =\sum\limits_{n=1}^{ck}\chi\left(  n\right)
\mathfrak{B}_{1,\chi}\left(  \frac{dn}{c}\right)  \mathfrak{B}_{1}\left(
\frac{n}{ck}\right)
\]
and possesses the reciprocity formula \cite[Theorem 4]{b2}%
\[
s\left(  c,d:\chi\right)  +s\left(  d,c:\bar{\chi}\right)  =B_{1,\chi
}B_{1,\bar{\chi}},
\]
whenever $d$ and $c$ are coprime positive integers, and either $c$ or
$d\equiv0\left(  \operatorname{mod}k\right)  $. Here $\mathfrak{B}_{p,\chi
}\left(  x\right)  $ are the generalized Bernoulli functions (see (\ref{3})
below) and $B_{p,\chi}=\mathfrak{B}_{p,\chi}\left(  0\right)  $. The sum
$s\left(  d,c:\chi\right)  $ is generalized, in the sense of Apostol, by%
\[
s_{p}\left(  d,c:\chi\right)  =\sum\limits_{n=1}^{ck}\chi\left(  n\right)
\mathfrak{B}_{p,\chi}\left(  \frac{dn}{c}\right)  \mathfrak{B}_{1}\left(
\frac{n}{ck}\right)
\]
and corresponding reciprocity formula is established \cite{cck}.

Generalizations of Hardy--Berndt sums $S(d,c),$ $s_{3}(d,c)$ and $s_{4}(d,c)$,
in the sense of $s_{p}\left(  d,c:\chi\right)  ,$ are presented in \cite{ck}
by obtaining transformation formulas for the function
\begin{equation}
B\left(  z,s:\chi\right)  =\sum\limits_{n=0}^{\infty}\sum\limits_{m=1}%
^{\infty}\chi(m)\chi(2n+1)\left(  2n+1\right)  ^{s-1}e^{\pi im(2n+1)z/k},
\label{B1}%
\end{equation}
which is a character extension of $\log\theta_{4}\left(  z\right)  .$

Inspiring by \cite{b2,ck} and the fact\textbf{\ }%
\[
\log\left(  \frac{\theta_{2}\left(  z\right)  }{2e^{\pi iz/4}}\right)
=-\sum\limits_{n=1}^{\infty}\sum\limits_{m=1}^{\infty}\left(  -1\right)
^{m}n^{-1}e^{2\pi imnz}%
\]
we set the function $A_{1}\left(  z,s:\chi\right)  $\ to be
\[
A_{1}\left(  z,s:\chi\right)  =\sum\limits_{n=1}^{\infty}\sum\limits_{m=1}%
^{\infty}\left(  -1\right)  ^{m}\chi\left(  m\right)  \chi\left(  n\right)
n^{s-1}e^{2\pi imnz/k},
\]
for $\operatorname{Im}\left(  z\right)  >0$\ and for all $s.$

In this paper, we derive transformation formulas for the function
$A_{1}\left(  z,s:\chi\right)  $. Sums appearing in transformation formulas
are generalizations, involving characters and generalized Bernoulli and Euler
functions, of Hardy--Berndt sums $s_{1}(d,c),$\ $s_{2}(d,c)$\ and
$s_{5}(d,c).$\ These new sums still obey reciprocity formulas.

\section{Preliminaries}

Throughout this paper $\chi$ denotes a non-principal primitive character of
modulus $k$. The letter $p$ always denotes positive integer. We use the
modular transformation $\left(  az+b\right)  /\left(  cz+d\right)  $ where
$a,$ $b,$ $c$ and $d$ are integers with $ad-bc=1$ and $c>0$. The upper
half-plane $\left\{  x+iy:y>0\right\}  $ will be denoted by $\mathbb{H}$ and
the upper quarter-plane $\left\{  x+iy:x>-d/c\text{, }y>0\right\}  $ by
$\mathbb{K}$. We use the notation $\left\{  x\right\}  $ for the fractional
part of $x.$ Unless otherwise stated, we assume that the branch of the
argument is defined by $-\pi\leq$ arg $z<\pi$.

The Bernoulli polynomials $B_{n}(x)$ and the Euler polynomials $E_{n}(x)$ are
defined by means of the generating functions
\[
\frac{te^{xt}}{e^{t}-1}=\sum\limits_{n=0}^{\infty}B_{n}(x)\frac{t^{n}}%
{n!},\ |t|<2\pi,\text{ and }\frac{2e^{xt}}{e^{t}+1}=\sum\limits_{n=0}^{\infty
}E_{n}(x)\frac{t^{n}}{n!},\ |t|<\pi,
\]
respectively (see \cite{j}). $B_{n}(0)=B_{n}$ are the Bernoulli numbers with
$B_{0}=1,$ $B_{1}=-1/2$ and $B_{2n-1}\left(  1/2\right)  =B_{2n+1}=0$ for
$n\geq1.$ For $0\leq x<1$ and $m\in\mathbb{Z}$, the Bernoulli functions
$\mathfrak{B}_{n}\left(  x\right)  $ are defined by
\[
\mathfrak{B}\left(  x+m\right)  =B_{n}\left(  x\right)  \text{ when }%
n\not =1\text{ or }x\not =0,\text{ and }\mathfrak{B}_{1}\left(  m\right)
=\mathfrak{B}_{1}\left(  0\right)  =0
\]
and satisfy Raabe theorem for any $x$
\begin{equation}
\sum\limits_{j=0}^{r-1}\mathfrak{B}_{n}\left(  x+\frac{j}{r}\right)
=r^{1-n}\mathfrak{B}_{n}\left(  rx\right)  . \label{1}%
\end{equation}
Also we have \cite[Eq. (4.5)]{cl}%
\begin{equation}
r^{n-1}\sum\limits_{j=0}^{r-1}\left(  -1\right)  ^{j}\mathfrak{B}_{n}\left(
\frac{x+j}{r}\right)  =-\frac{n}{2}\mathcal{E}_{n-1}\left(  x\right)
\label{11}%
\end{equation}
for even $r$ and any $x.$ Here $\mathcal{E}_{n}\left(  x\right)  $\ are the
Euler functions defined by
\begin{equation}
\mathcal{E}_{n}\left(  x\right)  =E_{n}\left(  x\right)  \text{ and
}\mathcal{E}_{n}\left(  x+m\right)  =\left(  -1\right)  ^{m}\mathcal{E}%
_{n}\left(  x\right)  \label{e1}%
\end{equation}
for $0\leq x<1$ and $m\in\mathbb{Z}$. The generalized Bernoulli function
$\mathfrak{B}_{m,\chi}\left(  x\right)  $ are defined by Berndt \cite{b5}.\ We
will often use the following property that can confer as a definition
\begin{equation}
\mathfrak{B}_{m,\chi}\left(  x\right)  =k^{m-1}\sum_{j=0}^{k-1}\bar{\chi
}\left(  j\right)  \mathfrak{B}_{m}\left(  \frac{j+x}{k}\right)  ,\ m\geq1,
\label{3}%
\end{equation}
and satisfy%
\begin{equation}
\mathfrak{B}_{m,\chi}\left(  x+nk\right)  =\mathfrak{B}_{m,\chi}\left(
x\right)  ,\text{ }\mathfrak{B}_{m,\chi}\left(  -x\right)  =\left(  -1\right)
^{m}\chi\left(  -1\right)  \mathfrak{B}_{m,\chi}\left(  x\right)  . \label{b}%
\end{equation}

For the convenience with the definition of $\mathfrak{B}_{m,\chi}\left(
x\right)  ,$ let the character Euler function $\mathcal{E}_{m,\chi}\left(
x\right)  $ be defined by
\begin{equation}
\mathcal{E}_{m,\chi}\left(  x\right)  =k^{m}\sum_{j=0}^{k-1}\left(  -1\right)
^{j}\bar{\chi}\left(  j\right)  \mathcal{E}_{m}\left(  \frac{j+x}{k}\right)
,\ m\geq0 \label{13a}%
\end{equation}
for odd $k,$ the modulus of $\chi.$ It is easily seen that
\begin{equation}
\mathcal{E}_{m,\chi}\left(  x+nk\right)  =\left(  -1\right)  ^{n}%
\mathcal{E}_{m,\chi}\left(  x\right)  ,\text{ }\mathcal{E}_{m,\chi}\left(
-x\right)  =\left(  -1\right)  ^{m-1}\chi\left(  -1\right)  \mathcal{E}%
_{m,\chi}\left(  x\right)  . \label{e}%
\end{equation}

The Gauss sum $G\left(  z,\chi\right)  $ is defined by%
\[
G\left(  z,\chi\right)  =\sum_{v=0}^{k-1}\chi\left(  v\right)  e^{2\pi
ivz/k}.
\]
We put $G\left(  1,\chi\right)  =G\left(  \chi\right)  .$ If $n$ is an
integer, then \cite[p. 168]{a2}%
\[
G\left(  n,\chi\right)  =\bar{\chi}\left(  n\right)  G\left(  \chi\right)  .
\]

Let $r_{1}$ and $r_{2}$ be arbitrary real numbers. For $z\in\mathbb{H}$ and
$Re\left(  s\right)  >2,$ Berndt \cite{b2} defines the function%
\begin{equation}
G\left(  z,s:\chi:r_{1},r_{2}\right)  =\sum\limits_{m,n=-\infty}^{\infty
}\ \hspace{-0.19in}^{^{\prime}}\ \frac{\chi(m)\bar{\chi}(n)}{\left(  \left(
m+r_{1}\right)  z+n+r_{2}\right)  ^{s}}, \label{7a}%
\end{equation}
where the dash means that the possible pair $m=-r_{1},$ $n=-r_{2}$ is omitted
from the summation. In accordance with the subject of this study we present
Berndt's formulas for $r_{1}=r_{2}=0.$ Set $G\left(  z,s:\chi\right)
=G\left(  z,s:\chi:0,0\right)  $ and
\[
A\left(  z,s:\chi\right)  =\sum\limits_{m=1}^{\infty}\chi(m)\sum
\limits_{n=1}^{\infty}\chi(n)n^{s-1}e^{2\pi inmz/k},\text{ }z\in
\mathbb{H}\text{ and }s\in\mathbb{C}.
\]
Then, it is shown that
\[
\Gamma\left(  s\right)  G\left(  z,s:\chi\right)  =G(\bar{\chi})\left(
-\frac{2\pi i}{k}\right)  ^{s}H\left(  z,s:\chi\right)
\]
where $H\left(  z,s:\chi\right)  =\left(  1+e^{\pi is}\right)  A\left(
z,s:\chi\right)  .$

The following lemma is due to Lewittes \cite[Lemma 1]{lew}.

\begin{lemma}
\label{le1} Let $A,$ $B,$ $C$ and $D$ be real with $A$ and $B$ not both zero
and $C>0$. Then for $z\in\mathbb{H},$%
\[
arg\left(  \left(  Az+B\right)  /\left(  Cz+D\right)  \right)  =arg\left(
Az+B\right)  -arg\left(  Cz+D\right)  +2\pi l,
\]
where $l$ is independent of $z$ and $l=\left\{
\begin{array}
[c]{ll}%
1, & A\leq0\text{ and }AD-BC>0,\\
0, & \text{otherwise.}%
\end{array}
\right.  $
\end{lemma}

We need the following Berndt's transformation formulas (see \cite[Theorem
1]{dc2} and \cite[Theorem 2]{m3} for generalizations).

\begin{theorem}
\label{tk1}\cite[Theorem 2]{b2} Let $Tz=\left(  az+b\right)  /\left(
cz+d\right)  $.\ Suppose first that $a\equiv d\equiv0(mod$ $k).$ Then for
$z\in\mathbb{K}$ and $s\in\mathbb{C},$
\begin{align*}
&  \left(  cz+d\right)  ^{-s}\Gamma\left(  s\right)  G\left(  Tz,s:\chi
\right)  =\bar{\chi}(b)\chi(c)\Gamma\left(  s\right)  G\left(  z,s:\bar{\chi
}\right) \\
&  \quad+\bar{\chi}(b)\chi(c)e^{-\pi is}\sum\limits_{j=1}^{c}\sum
\limits_{\mu=0}^{k-1}\sum\limits_{\nu=0}^{k-1}\bar{\chi}\left(  \mu
c+j\right)  \chi\left(  \left[  \tfrac{dj}{c}\right]  -\nu\right)  f(z,s,c,d),
\end{align*}
where
\begin{equation}
f(z,s,c,d)=\int\limits_{C}\frac{e^{-\left(  \mu c+j\right)  \left(
cz+d\right)  u/c}}{e^{-\left(  cz+d\right)  ku}-1}\frac{e^{\left(
\nu+\left\{  dj/c\right\}  \right)  u}}{e^{ku}-1}u^{s-1}du, \label{2}%
\end{equation}
where $C$ is a loop beginning at $+\infty,$ proceeding in the upper
half-plane, encircling the origin in the positive direction so that $u=0$ is
the only zero of $\left(  e^{-\left(  cz+d\right)  ku}-1\right)  \left(
e^{ku}-1\right)  $ lying \textquotedblleft inside" the loop, and then
returning to $+\infty$ in the lower half-plane. Here we choose the branch of
$u^{s}$ with $0<arg$ $u<2\pi.$

Secondly, if $b\equiv c\equiv0(mod$ $k),$ we have for $z\in\mathbb{K}$ and
$s\in\mathbb{C},$
\begin{align*}
&  \left(  cz+d\right)  ^{-s}\Gamma\left(  s\right)  G\left(  Tz,s:\chi
\right)  =\bar{\chi}(a)\chi(d)\Gamma\left(  s\right)  G\left(  z,s:\chi\right)
\\
&  \quad+\bar{\chi}(a)\chi(d)e^{-\pi is}\sum\limits_{j=1}^{c}\sum
\limits_{\mu=0}^{k-1}\sum\limits_{\nu=0}^{k-1}\chi\left(  j\right)  \bar{\chi
}\left(  \left[  \tfrac{dj}{c}\right]  +d\mu-\nu\right)  f(z,s,c,d).
\end{align*}

\end{theorem}

\section{Transformation Formulas}

In the sequel, unless otherwise stated, we assume that $k$ is odd.

From definition, $A_{1}\left(  z,s:\chi\right)  $ can be written in terms of
$A\left(  z,s:\chi\right)  $ as
\[
A_{1}\left(  z,s:\chi\right)  =2\chi\left(  2\right)  A\left(  2z,s:\chi
\right)  -A\left(  z,s:\chi\right)  .
\]
Thus, transformation formulas can be achieved for the function $H_{1}\left(
z,s:\chi\right)  =\left(  1+e^{\pi is}\right)  A_{1}\left(  z,s:\chi\right)  $
with the help of Theorem \ref{tk1}. We have following transformation formulas
according to $d$ or $c$ is even.

\begin{theorem}
\label{teod2}Let $Tz=\left(  az+b\right)  /\left(  cz+d\right)  $ and $d$ be
even. If $a\equiv d\equiv0$ $(\operatorname{mod}k),$ then for $z\in\mathbb{K}$
and $s\in\mathbb{C}$%
\begin{align}
&  G\left(  \bar{\chi}\right)  \left(  cz+d\right)  ^{-s}H_{1}\left(
Tz,s:\chi\right)  =\bar{\chi}\left(  b\right)  \chi\left(  c\right)  G\left(
\chi\right)  2^{1-s}\chi\left(  2\right)  B_{1}\left(  z,s:\bar{\chi}\right)
\nonumber\\
&  \ +\bar{\chi}\left(  b\right)  \chi\left(  c\right)  \left(  -\frac{k}{2\pi
i}\right)  ^{s}e^{-\pi is}\sum\limits_{j=1}^{c}\sum\limits_{\mu=0}^{k-1}%
\sum\limits_{\nu=0}^{k-1}\bar{\chi}\left(  \mu c+j\right) \nonumber\\
&  \quad\times\left\{  \frac{\chi\left(  2\right)  }{2^{s-1}}\chi\left(
\left[  \frac{dj}{2c}\right]  -\nu\right)  f\left(  \frac{z}{2},s,c,\frac
{d}{2}\right)  -\chi\left(  \left[  \frac{dj}{c}\right]  -\nu\right)  f\left(
z,s,c,d\right)  \right\}  , \label{d1-a}%
\end{align}
where $B_{1}\left(  z,s:\chi\right)  =\left(  1+e^{\pi is}\right)  B\left(
z,s:\chi\right)  $ is given by (\ref{B1}) and $f\left(  z,s,c,d\right)  $ is
given by (\ref{2}). If $b\equiv c\equiv0$ $(\operatorname{mod}k)$, then%
\begin{align*}
&  G\left(  \bar{\chi}\right)  \left(  cz+d\right)  ^{-s}H_{1}\left(
Tz,s:\chi\right)  =\bar{\chi}\left(  a\right)  \chi\left(  d\right)  G\left(
\bar{\chi}\right)  2^{1-s}\bar{\chi}\left(  2\right)  B_{1}\left(
z,s:\chi\right) \\
&  \ +\bar{\chi}\left(  a\right)  \chi\left(  d\right)  \left(  -\frac{k}{2\pi
i}\right)  ^{s}e^{-\pi is}\sum_{j=1}^{c}\sum_{\mu=0}^{k-1}\sum_{\nu=0}%
^{k-1}\chi\left(  j\right) \\
&  \quad\times\left\{  2^{1-s}\bar{\chi}\left(  \left[  \frac{dj}{2c}\right]
+\frac{d}{2}\mu-\nu\right)  f\left(  \frac{z}{2},s,c,\frac{d}{2}\right)
-\bar{\chi}\left(  \left[  \frac{dj}{c}\right]  +d\mu-\nu\right)  f\left(
z,s,c,d\right)  \right\}  .
\end{align*}

\end{theorem}

\begin{proof}
For even $d,$ let $Sz=\left(  2az+b\right)  /\left(  cz+d/2\right)  .$ Since
$S\left(  z/2\right)  =2T\left(  z\right)  ,$ one can write%
\begin{equation}
H_{1}\left(  Tz,s:\chi\right)  =2\chi\left(  2\right)  H\left(  S\left(
z/2\right)  ,s:\chi\right)  -H\left(  Tz,s:\chi\right)  . \label{H2}%
\end{equation}
Thus, the desired result follows from (\ref{H2}) and Theorem \ref{tk1} with
\cite[Eq. (3.1)]{ck}
\[
2^{1-s}\chi\left(  2\right)  H\left(  \frac{z}{2},s:\bar{\chi}\right)
-H\left(  z,s:\bar{\chi}\right)  =2^{1-s}\chi\left(  2\right)  B_{1}\left(
z,s:\bar{\chi}\right)  .
\]

\end{proof}

\begin{theorem}
\label{teod1}Let $Tz=\left(  az+b\right)  /\left(  cz+d\right)  $ and let $c$
be even. If $a\equiv d\equiv0$ $(\operatorname{mod}k),$ then for
$z\in\mathbb{K}$ and all $s\in\mathbb{C}$%
\begin{align*}
&  G\left(  \bar{\chi}\right)  \left(  cz+d\right)  ^{-s}H_{1}\left(
Tz\,,s:\chi\right)  =\bar{\chi}\left(  b\right)  \chi\left(  c\right)
G\left(  \chi\right)  H_{1}\left(  z,s:\bar{\chi}\right) \\
&  \quad-\bar{\chi}\left(  b\right)  \chi\left(  c\right)  \left(  -\frac
{k}{2\pi i}\right)  ^{s}e^{-\pi is}\sum\limits_{\mu=0}^{k-1}\sum
\limits_{\nu=0}^{k-1}\\
&  \quad\quad\times\left\{  \sum\limits_{j=1}^{c}\bar{\chi}\left(  \mu
c+j\right)  \chi\left(  \left[  \frac{dj}{c}\right]  -\nu\right)  f\left(
z,s,c,d\right)  \right. \\
&  \qquad\qquad\left.  -\sum\limits_{j=1}^{c/2}2\bar{\chi}\left(  2\right)
\bar{\chi}\left(  \frac{\mu c}{2}+j\right)  \chi\left(  \left[  \frac{2dj}%
{c}\right]  -\nu\right)  f\left(  2z,s,\frac{c}{2},d\right)  \right\}  .
\end{align*}
If $b\equiv c\equiv0(\operatorname{mod}k),$ then%
\begin{align*}
&  \left(  cz+d\right)  ^{-s}G\left(  \bar{\chi}\right)  H_{1}\left(
Tz,s:\chi\right)  =\bar{\chi}\left(  a\right)  \chi\left(  d\right)  G\left(
\bar{\chi}\right)  H_{1}\left(  z,s:\chi\right) \\
&  \quad-\bar{\chi}\left(  a\right)  \chi\left(  d\right)  \left(  -\frac
{k}{2\pi i}\right)  ^{s}e^{-\pi is}\sum_{\mu=0}^{k-1}\sum_{\nu=0}^{k-1}\\
&  \quad\quad\times\left.  \sum_{j=1}^{c^{\text{ }}}\chi\left(  j\right)
\bar{\chi}\left(  \left[  \frac{dj}{c}\right]  +d\mu-\nu\right)  f\left(
z,s,c,d\right)  \right. \\
&  \qquad\qquad\left.  -\sum_{j=1}^{c/2}2\chi\left(  2\right)  \chi\left(
j\right)  \bar{\chi}\left(  \left[  \frac{2dj}{c}\right]  +d\mu-\nu\right)
f\left(  2z,s,\frac{c}{2},d\right)  \right\}  .
\end{align*}

\end{theorem}

\begin{proof}
For even $c,$ if we set $Vz=\left(  az+2b\right)  /\left(  \dfrac{c}%
{2}z+d\right)  ,$ then $V\left(  2z\right)  =2T\left(  z\right)  $ and \
\begin{equation}
H_{1}\left(  Tz,s:\chi\right)  =2\chi\left(  2\right)  H\left(  V\left(
2z\right)  ,s:\chi\right)  -H\left(  Tz,s:\chi\right)  . \label{H1}%
\end{equation}
Using (\ref{H1}) and Theorem \ref{tk1} completes the proof.
\end{proof}

Theorem \ref{teod2} and Theorem \ref{teod1} can be simplified when $s=1-p$ is
an integer for $p\geq1$. In this case, by the residue theorem, we have
\begin{align}
f\left(  z,1-p,c,d\right)   &  =\frac{2\pi ik^{p-1}}{\left(  p+1\right)
!}\sum\limits_{m=0}^{p+1}\binom{p+1}{m}\left(  -\left(  cz+d\right)  \right)
^{m-1}\nonumber\\
&  \ \quad\times B_{p+1-m}\left(  \frac{\nu+\left\{  \frac{dj}{c}\right\}
}{k}\right)  B_{m}\left(  \frac{\mu c+j}{ck}\right)  . \label{fp}%
\end{align}

The following is character extension of (\ref{6d}).

\begin{theorem}
\label{teoS1} Let $p\geq1$ be odd and $d$ be even. If $a\equiv d\equiv0$
$(\operatorname{mod}k),$ then for $z\in\mathbb{H}$
\begin{align}
&  G\left(  \bar{\chi}\right)  \left(  cz+d\right)  ^{p-1}H_{1}\left(
Tz,1-p:\chi\right) \nonumber\\
&  =\bar{\chi}\left(  b\right)  \chi\left(  c\right)  \left(  2^{p}\chi\left(
2\right)  G\left(  \chi\right)  B_{1}\left(  z,1-p:\bar{\chi}\right)
-\frac{\chi\left(  -1\right)  \left(  2\pi i\right)  ^{p}}{2\left(  p!\right)
}g_{1}\left(  c,d,z,p,\bar{\chi}\right)  \right)  , \label{ad1}%
\end{align}
where%
\begin{align}
&  g_{1}\left(  c,d,z,p,\chi\right) \nonumber\\
&  \ =\sum_{m=1}^{p}\binom{p}{m}k^{m-p}\left(  -\left(  cz+d\right)  \right)
^{m-1}\sum_{n=1}^{ck}\chi\left(  n\right)  \mathcal{E}_{p-m,\chi}\left(
\frac{dn}{c}\right)  \mathfrak{B}_{m}\left(  \frac{n}{ck}\right)  . \label{g1}%
\end{align}
If $b\equiv c\equiv0$ $(\operatorname{mod}k),$ then for $z\in\mathbb{H}$%
\begin{align*}
&  G\left(  \bar{\chi}\right)  \left(  cz+d\right)  ^{p-1}H_{1}\left(
Tz,1-p:\chi\right) \\
&  =\bar{\chi}\left(  a\right)  \chi\left(  d\right)  \left(  2^{p}\bar{\chi
}\left(  2\right)  G\left(  \bar{\chi}\right)  B_{1}\left(  z,1-p:\chi\right)
-\frac{\chi\left(  -1\right)  \left(  2\pi i\right)  ^{p}}{2\left(  p!\right)
}g_{1}\left(  c,d,z,p,\chi\right)  \right)  .
\end{align*}

\end{theorem}

\begin{proof}
Let us consider the case $a\equiv d\equiv0\left(  \operatorname{mod}k\right)
.$ By aid of (\ref{fp}), equation (\ref{d1-a}) turns into
\begin{align*}
&  G\left(  \bar{\chi}\right)  \left(  cz+d\right)  ^{p-1}H_{1}\left(
Tz,1-p:\chi\right) \\
&  =\bar{\chi}\left(  b\right)  \chi\left(  c\right)  G\left(  \chi\right)
2^{p}\chi\left(  2\right)  B_{1}\left(  z,1-p:\bar{\chi}\right) \\
&  +\bar{\chi}\left(  b\right)  \chi\left(  c\right)  \frac{\left(  2\pi
i\right)  ^{p}}{\left(  p+1\right)  !}\sum_{m=1}^{p}\binom{p+1}{m}\left(
-\left(  cz+d\right)  \right)  ^{m-1}\left(  T_{1}-T_{2}\right)  ,
\end{align*}
where
\begin{align}
T_{1}  &  =2^{p+1-m}\chi\left(  2\right)  \sum_{j=1}^{c}\sum_{\mu=0}^{k-1}%
\sum_{\nu=0}^{k-1}\bar{\chi}\left(  \mu c+j\right)  \chi\left(  \left[
\frac{dj}{2c}\right]  -\nu\right) \nonumber\\
&  \qquad\times B_{p+1-m}\left(  \frac{\nu+\left\{  dj/2c\right\}  }%
{k}\right)  B_{m}\left(  \frac{\mu c+j}{ck}\right)  ,\label{T1}\\
T_{2}  &  =\sum_{j=1}^{c}\sum_{\mu=0}^{k-1}\sum_{\nu=0}^{k-1}\bar{\chi}\left(
\mu c+j\right)  \chi\left(  \left[  \frac{dj}{c}\right]  -\nu\right)
B_{p+1-m}\left(  \frac{\nu+\left\{  dj/c\right\}  }{k}\right)  B_{m}\left(
\frac{\mu c+j}{ck}\right)  \nonumber\label{T2}%
\end{align}
and we have used that the sum over $\mu$ is zero for $m=0$ and the sum over
$\nu$ is zero for $m=p+1$. We first note that the triple sum in (\ref{T1}) is
invariant by replacing $B_{p+1-m}\left(  \frac{\nu+\left\{  dj/2c\right\}
}{k}\right)  $ by $\mathfrak{B}_{p+1-m}\left(  \frac{\nu+\left\{
dj/2c\right\}  }{k}\right)  $ since $B_{p+1-m}\left(  \frac{\nu+\left\{
dj/2c\right\}  }{k}\right)  =\mathfrak{B}_{p+1-m}\left(  \frac{\nu+\left\{
dj/2c\right\}  }{k}\right)  $ for $0<\frac{\nu+\left\{  dj/2c\right\}  }%
{k}<1,$ and $\chi\left(  d/2\right)  =0 $ ($d\equiv0(\operatorname{mod}k)$ and
$k$ is odd) for $\frac{\nu+\left\{  dj/2c\right\}  }{k}=0$. Similarly, one can
write $\mathfrak{B}_{m}\left(  \frac{\mu c+j}{ck}\right)  $ in place of
$B_{m}\left(  \frac{\mu c+j}{ck}\right)  .$ After some manipulations, we see
that%
\[
T_{1}=2^{p+1-m}\chi\left(  -2\right)  k^{m-p}\sum_{n=1}^{ck}\bar{\chi}\left(
n\right)  \mathfrak{B}_{p+1-m,\bar{\chi}}\left(  \frac{dn}{2c}\right)
\mathfrak{B}_{m}\left(  \frac{n}{ck}\right)
\]
and%
\[
T_{2}=\chi\left(  -1\right)  k^{m-p}\sum_{n=1}^{ck}\bar{\chi}\left(  n\right)
\mathfrak{B}_{p+1-m,\bar{\chi}}\left(  \frac{dn}{c}\right)  \mathfrak{B}%
_{m}\left(  \frac{n}{ck}\right)  .
\]
So, we have%
\begin{align}
&  G\left(  \bar{\chi}\right)  \left(  cz+d\right)  ^{p-1}H_{1}\left(
Tz,1-p:\chi\right) \nonumber\\
&  =\bar{\chi}\left(  b\right)  \chi\left(  c\right)  G\left(  \chi\right)
2^{p}\chi\left(  2\right)  B_{1}\left(  z,1-p:\bar{\chi}\right) \nonumber\\
&  \quad+\bar{\chi}\left(  b\right)  \chi\left(  -c\right)  \frac{\left(  2\pi
i\right)  ^{p}}{\left(  p+1\right)  !}\sum\limits_{m=1}^{p}\binom{p+1}%
{m}k^{m-p}\left(  -\left(  cz+d\right)  \right)  ^{m-1}\nonumber\\
&  \qquad\times\sum_{n=1}^{ck}\bar{\chi}\left(  n\right)  \mathfrak{B}%
_{m}\left(  \frac{n}{ck}\right)  \left(  2^{p+1-m}\chi\left(  2\right)
\mathfrak{B}_{p+1-m,\bar{\chi}}\left(  \frac{dn}{2c}\right)  -\mathfrak{B}%
_{p+1-m,\bar{\chi}}\left(  \frac{dn}{c}\right)  \right)  . \label{TT}%
\end{align}
Now, consider the difference\textbf{\ }%
\[
T_{3}=2^{p+1-m}\chi\left(  2\right)  \mathfrak{B}_{p+1-m,\bar{\chi}}\left(
\frac{dn}{2c}\right)  -\mathfrak{B}_{p+1-m,\bar{\chi}}\left(  \frac{dn}%
{c}\right)  .
\]
Using the property \cite[Eq. (3.13)]{ck}%
\begin{equation}
\sum_{j=0}^{r-1}\mathfrak{B}_{m,\chi}\left(  x+\frac{jk}{r}\right)
=\chi(r)r^{1-m}\mathfrak{B}_{m,\chi}\left(  rx\right)  \label{6}%
\end{equation}
for $r=2$ and $x=dn/2c$, and utilizing (\ref{11}) we find that%
\begin{align}
T_{3}  &  =\left(  2k\right)  ^{p-m}\sum_{\mu=0}^{k-1}\chi\left(  2\mu\right)
\left(  \mathfrak{B}_{p+1-m}\left(  \frac{\mu+dn/2c}{k}\right)  -\mathfrak{B}%
_{p+1-m}\left(  \frac{\mu+dn/2c}{k}+\frac{1}{2}\right)  \right) \nonumber\\
&  =-\frac{p+1-m}{2}k^{p-m}\sum_{\mu=0}^{k-1}\chi\left(  2\mu\right)
\mathcal{E}_{p-m}\left(  \frac{2\mu+dn/c}{k}\right)  . \label{T3}%
\end{align}
The sum in the last line can be evaluated as%
\begin{align}
&  \sum_{\mu=0}^{\left(  k-1\right)  /2}\chi\left(  2\mu\right)
\mathcal{E}_{m}\left(  \frac{2\mu+x}{k}\right)  +\sum_{\mu=\left(  k+1\right)
/2}^{k-1}\chi\left(  2\mu\right)  \mathcal{E}_{m}\left(  \frac{2\mu+x}%
{k}\right) \nonumber\\
&  =\sum_{\mu=0}^{\frac{k-1}{2}}\chi\left(  2\mu\right)  \mathcal{E}%
_{m}\left(  \frac{2\mu+x}{k}\right)  -\sum_{\mu=0}^{\frac{k-3}{2}}\chi\left(
2\mu+1\right)  \mathcal{E}_{m}\left(  \frac{2\mu+1+x}{k}\right) \nonumber\\
&  =\sum_{\mu=0}^{k-1}\left(  -1\right)  ^{\mu}\chi\left(  \mu\right)
\mathcal{E}_{m}\left(  \frac{\mu+x}{k}\right)  . \label{12}%
\end{align}
Setting this in (\ref{T3}) and using (\ref{13a}) give%
\begin{equation}
2^{p+1-m}\chi\left(  2\right)  \mathfrak{B}_{p+1-m,\bar{\chi}}\left(  \frac
{x}{2}\right)  -\mathfrak{B}_{p+1-m,\bar{\chi}}\left(  x\right)
=-\frac{p+1-m}{2}\mathcal{E}_{p-m,\bar{\chi}}\left(  x\right)  . \label{be}%
\end{equation}
Therefore, we arrive at%
\begin{align*}
&  G\left(  \bar{\chi}\right)  \left(  cz+d\right)  ^{p-1}H_{1}\left(
Tz,1-p:\chi\right) \\
&  =\bar{\chi}\left(  b\right)  \chi\left(  c\right)  2^{p}\chi\left(
2\right)  G\left(  \chi\right)  B_{1}\left(  z,1-p:\bar{\chi}\right) \\
&  -\bar{\chi}\left(  b\right)  \chi\left(  -c\right)  \frac{\left(  2\pi
i\right)  ^{p}}{\left(  p+1\right)  !}\sum_{m=1}^{p}\frac{p+1-m}{2}\binom
{p+1}{m}k^{m-p}\left(  -\left(  cz+d\right)  \right)  ^{m-1}\\
&  \qquad\times\sum_{n=1}^{ck}\bar{\chi}\left(  n\right)  \mathcal{E}%
_{p-m,\bar{\chi}}\left(  \frac{dn}{c}\right)  \mathfrak{B}_{m}\left(  \frac
{n}{ck}\right)  .
\end{align*}
The result holds for $z\in\mathbb{H}$ by analytic continuation.

The proof for $b\equiv c\equiv0$ $\left(  \operatorname{mod}k\right)  $ is
completely analogous.
\end{proof}

Note that for odd $d$ and $c,$ if we take $Rz=T\left(  z+k\right)
=\dfrac{az+b+ak}{cz+d+ck}$ instead of\textbf{\ }$Tz=\dfrac{az+b}{cz+d}$ in
Theorem \ref{teoS1}, then the function $g_{1}\left(  c,d+ck,z,p,\chi\right)  $
turns into%
\begin{align}
g_{1}\left(  c,d+ck,z,p,\chi\right)   &  ={\sum\limits_{m=1}^{p}}\binom{p}%
{m}k^{m-p}\left(  -\left(  cz+d+ck\right)  \right)  ^{m-1}\nonumber\\
&  \quad\times{\sum\limits_{n=1}^{ck}}\left(  -1\right)  ^{n}\chi\left(
n\right)  \mathcal{E}_{p-m,\chi}\left(  \frac{dn}{c}\right)  \mathfrak{B}%
_{m}\left(  \frac{n}{ck}\right)  , \label{17}%
\end{align}
by (\ref{e}). So, it is convenient to present the following theorem since it
is observed a new sum.

\begin{theorem}
\label{c5} Let $p\geq1$ be odd and $Rz=\left(  az+b+ak\right)  /\left(
cz+d+ck\right)  $ with $d$ and $c$\ odd. If $a\equiv d\equiv0$
$(\operatorname{mod}k), $ for $z\in\mathbb{H}$,%
\begin{align}
&  G\left(  \bar{\chi}\right)  \left(  cz+d+ck\right)  ^{p-1}H_{1}\left(
Rz,1-p:\chi\right) \label{14}\\
&  =\bar{\chi}\left(  b\right)  \chi\left(  c\right)  G\left(  \chi\right)
2^{p}\chi\left(  2\right)  B_{1}\left(  z,1-p:\bar{\chi}\right)  -\frac
{\bar{\chi}\left(  b\right)  \chi\left(  -c\right)  }{2}\frac{\left(  2\pi
i\right)  ^{p}}{p!}g_{1}\left(  c,d+ck,z,p,\bar{\chi}\right)  ,\nonumber
\end{align}
where $g_{1}\left(  c,d+ck,z,p,\chi\right)  $ is given by (\ref{17}).

If $b\equiv c\equiv0$ $(\operatorname{mod}k),$ for $z\in\mathbb{H}$,%
\begin{align}
&  G\left(  \bar{\chi}\right)  \left(  cz+d+ck\right)  ^{p-1}H_{1}\left(
Rz,1-p:\chi\right) \label{18}\\
&  =\bar{\chi}\left(  a\right)  \chi\left(  d\right)  G\left(  \bar{\chi
}\right)  2^{p}\bar{\chi}\left(  2\right)  B_{1}\left(  z,1-p:\chi\right)
-\frac{\bar{\chi}\left(  a\right)  \chi\left(  -d\right)  }{2}\frac{\left(
2\pi i\right)  ^{p}}{p!}g_{1}\left(  c,d+ck,z,p,\chi\right)  .\nonumber
\end{align}

\end{theorem}

For $s=1-p,$ Theorem \ref{teod1} turns into following, which is character
analogous of (\ref{6e}).

\begin{theorem}
\label{teoS2} Let $p\geq1$ be odd and let $c$ be even. If $a\equiv d\equiv0$
$(\operatorname{mod}k),$ for $z\in\mathbb{H}$,
\begin{align}
&  G\left(  \bar{\chi}\right)  \left(  cz+d\right)  ^{p-1}H_{1}\left(
Tz,1-p:\chi\right) \nonumber\\
&  \ =\bar{\chi}\left(  b\right)  \chi\left(  c\right)  G\left(  \chi\right)
H_{1}\left(  z,1-p:\bar{\chi}\right)  +\bar{\chi}\left(  -b\right)
\chi\left(  c\right)  \frac{\left(  2\pi i\right)  ^{p}}{\left(  p+1\right)
!}g_{2}\left(  c,d,z,p,\bar{\chi}\right)  , \label{ad2}%
\end{align}
where%
\begin{align}
g_{2}\left(  c,d,z,p,\chi\right)   &  =\sum\limits_{m=1}^{p}\binom{p+1}%
{m}\left(  -\left(  cz+d\right)  \right)  ^{m-1}k^{m-p}\nonumber\\
&  \quad\times\sum\limits_{n=1}^{ck}\left(  -1\right)  ^{n}\chi\left(
n\right)  \mathfrak{B}_{p+1-m,\chi}\left(  \frac{dn}{c}\right)  \mathfrak{B}%
_{m}\left(  \frac{n}{ck}\right)  . \label{g2}%
\end{align}
If $b\equiv c\equiv0$ $(\operatorname{mod}k),$ for $z\in\mathbb{H}$,%
\begin{align}
&  G\left(  \bar{\chi}\right)  \left(  cz+d\right)  ^{p-1}H_{1}\left(
Tz,1-p:\chi\right) \nonumber\\
&  \ =\bar{\chi}\left(  a\right)  \chi\left(  d\right)  G\left(  \bar{\chi
}\right)  H_{1}\left(  z,1-p:\chi\right)  +\bar{\chi}\left(  a\right)
\chi\left(  -d\right)  \frac{\left(  2\pi i\right)  ^{p}}{\left(  p+1\right)
!}g_{2}\left(  c,d,z,p,\chi\right)  . \label{bc2}%
\end{align}

\end{theorem}

\begin{proof}
Similar to (\ref{TT}), it can be found that
\begin{align*}
&  G\left(  \bar{\chi}\right)  \left(  cz+d\right)  ^{p-1}H_{1}\left(
Tz,1-p:\chi\right) \\
&  =\bar{\chi}\left(  b\right)  \chi\left(  c\right)  G\left(  \chi\right)
H_{1}\left(  z,1-p:\bar{\chi}\right) \\
&  \quad-\bar{\chi}\left(  -b\right)  \chi\left(  c\right)  \frac{\left(  2\pi
i\right)  ^{p}}{\left(  p+1\right)  !}\sum\limits_{m=1}^{p}\binom{p+1}%
{m}\left(  -\left(  cz+d\right)  \right)  ^{m-1}k^{m-p}\\
&  \quad\quad\times\left(  \sum\limits_{n=1}^{ck}\bar{\chi}\left(  n\right)
\mathfrak{B}_{p+1-m,\bar{\chi}}\left(  \frac{dn}{c}\right)  \mathfrak{B}%
_{m}\left(  \frac{n}{ck}\right)  \right. \\
&  \quad\qquad\left.  -2\sum_{n=1}^{ck/2}\bar{\chi}\left(  2n\right)
\mathfrak{B}_{p+1-m,\bar{\chi}}\left(  \frac{2dn}{c}\right)  \mathfrak{B}%
_{m}\left(  \frac{2n}{ck}\right)  \right)  .
\end{align*}
Then, (\ref{ad2}) follows from
\begin{align*}
&  2\sum\limits_{n=1}^{ck/2}\bar{\chi}\left(  2n\right)  \mathfrak{B}%
_{p+1-m,\bar{\chi}}\left(  \frac{2dn}{c}\right)  \mathfrak{B}_{m}\left(
\frac{2n}{ck}\right)  -\sum\limits_{n=1}^{ck}\bar{\chi}\left(  n\right)
\mathfrak{B}_{p+1-m,\bar{\chi}}\left(  \frac{dn}{c}\right)  \mathfrak{B}%
_{m}\left(  \frac{n}{ck}\right) \\
&  =\sum\limits_{n=1}^{ck}\left(  -1\right)  ^{n}\bar{\chi}\left(  n\right)
\mathfrak{B}_{p+1-m,\bar{\chi}}\left(  \frac{dn}{c}\right)  \mathfrak{B}%
_{m}\left(  \frac{n}{ck}\right)  .
\end{align*}

The proof for $b\equiv c\equiv0(\operatorname{mod}k)$ is completely analogous.
\end{proof}

Note that for $p=1$ transformation formulas (\ref{ad2}) and (\ref{bc2})
coincide with Meyer's \cite{m3} second formulas in Theorems 10 and 11, respectively.

\section{Reciprocity Theorems}

In this section, we first give reciprocity formulas for the functions
$g_{1}(d,c+dk,z,p,\chi),$ $g_{1}(d,c,z,p,\chi)$ and $g_{2}(d,c,z,p,\chi).$ In
particular, these formulas yield reciprocity formulas analogues to the
reciprocity formulas (\ref{8}) and (\ref{9}).

We need the following theorem, offered by Can and Kurt \cite{ck}.

\begin{theorem}
(see \cite[Eqs. (3.4) and (3.20)]{ck}) Let $p\geq1$ be odd integer and
$Tz=\left(  az+b\right)  /\left(  cz+d\right)  .$ If $b$ is even and $a\equiv
d\equiv0\left(  \operatorname{mod}k\right)  ,$ then for $z\in\mathbb{H}$%
\begin{align}
&  \left(  cz+d\right)  ^{p-1}G\left(  \bar{\chi}\right)  B_{1}\left(
Tz,1-p:\chi\right) \nonumber\\
&  =\bar{\chi}\left(  \frac{b}{2}\right)  \chi\left(  2c\right)  G\left(
\chi\right)  B_{1}\left(  z,1-p:\bar{\chi}\right) \nonumber\\
&  \quad-\bar{\chi}\left(  \frac{b}{2}\right)  \chi\left(  -2c\right)
\frac{\left(  2\pi i\right)  ^{p}}{\left(  p+1\right)  !}{\displaystyle\sum
\limits_{m=1}^{p}} \binom{p+1}{m}k^{m-p}\left(  -\left(  cz+d\right)  \right)
^{m-1}\nonumber\\
&  \qquad\times\frac{m}{2^{m}}{\displaystyle\sum\limits_{n=1}^{ck}}
\chi\left(  n\right)  \mathfrak{B}_{p+1-m,\chi}\left(  \frac{dn}{2c}\right)
\mathcal{E}_{m-1}\left(  \frac{n}{ck}\right)  . \label{B1b}%
\end{align}
If $a$ is even and $a\equiv d\equiv0\left(  \operatorname{mod}k\right)  ,$
then for $z\in\mathbb{H}$
\begin{align}
&  2^{p}\bar{\chi}(2)\left(  cz+d\right)  ^{p-1}G\left(  \bar{\chi}\right)
B_{1}\left(  Tz,1-p:\chi\right) \nonumber\\
&  \ =\bar{\chi}(b)\chi(c)G(\chi)H_{1}\left(  z,1-p:\bar{\chi}\right)
\nonumber\\
&  \quad+\bar{\chi}(b)\chi(-c)\frac{\left(  2\pi i\right)  ^{p}}{\left(
p+1\right)  !}\sum\limits_{m=1}^{p+1}\binom{p+1}{m}\left(  -\left(
cz+d\right)  \right)  ^{m-1}k^{m-p}\nonumber\\
&  \qquad\times\left(  -\frac{m}{2}\sum\limits_{n=1}^{ck}\left(  -1\right)
^{n}\bar{\chi}(n)\mathfrak{B}_{p+1-m,\bar{\chi}}\left(  \frac{dn}{c}\right)
\mathcal{E}_{m-1}\left(  \frac{n}{ck}\right)  \right)  . \label{B1a}%
\end{align}

\end{theorem}

The function $g_{1}(d,c+dk,z,p,\chi),$ given by (\ref{17}), satisfies the
following reciprocity formula:

\begin{theorem}
\label{teoRP2}Let $p\geq1$ be odd and $d$ and $c$ be coprime integers. If $d$
or $c\equiv0$ $\left(  \operatorname{mod}k\right)  ,$ then%
\begin{align*}
&  g_{1}\left(  d,-c-dk,z,p,\chi\right)  -\chi\left(  -1\right)  \left(
z-k\right)  ^{p-1}g_{1}\left(  c,d+ck,V_{1}\left(  z\right)  ,p,\bar{\chi
}\right) \\
&  =\bar{\chi}\left(  4\right)  \frac{p}{2k^{p-1}}{\displaystyle\sum
\limits_{m=0}^{p-1}} \binom{p-1}{m}\left(  z-k\right)  ^{m}\mathcal{E}%
_{p-1-m,\bar{\chi}}\left(  0\right)  \mathcal{E}_{m,\chi}\left(  0\right)  ,
\end{align*}
where $V_{1}\left(  z\right)  =\dfrac{-kz+k^{2}-1}{z-k}.$
\end{theorem}

\begin{proof}
For even $\left(  c+d\right)  ,$ consider the modular substitutions $R\left(
z\right)  =\dfrac{az+b+ak}{cz+d+ck}$, $R^{\ast}\left(  z\right)
=\dfrac{bz-a-bk}{dz-c-dk}$ and $V_{1}\left(  z\right)  =\dfrac{-kz+k^{2}%
-1}{z-k}.$ Suppose $a\equiv d\equiv0(\operatorname{mod}k)$. Replacing $z$ by
$V_{1}\left(  z\right)  $ in (\ref{14}) gives%
\begin{align}
&  G\left(  \bar{\chi}\right)  \left(  \frac{dz-c-dk}{z-k}\right)  ^{p-1}%
H_{1}\left(  R^{\ast}\left(  z\right)  ,1-p:\chi\right) \nonumber\\
&  =\bar{\chi}\left(  b\right)  \chi\left(  c\right)  2^{p}\chi\left(
2\right)  G\left(  \chi\right)  B_{1}\left(  V_{1}\left(  z\right)
,1-p:\bar{\chi}\right) \nonumber\\
&  \quad-\frac{\bar{\chi}\left(  b\right)  \chi\left(  -c\right)  }{2}%
\frac{\left(  2\pi i\right)  ^{p}}{p!}g_{1}\left(  c,d+ck,V_{1}\left(
z\right)  ,p,\bar{\chi}\right)  . \label{A}%
\end{align}
Replacing $R\left(  z\right)  $ by $R^{\ast}\left(  z\right)  $ in (\ref{18})
yields%
\begin{align}
&  G\left(  \bar{\chi}\right)  \left(  dz-c-dk\right)  ^{p-1}H_{1}\left(
R^{\ast}\left(  z\right)  ,1-p:\chi\right) \nonumber\\
&  =\bar{\chi}\left(  b\right)  \chi\left(  c\right)  \left(  2^{p}\bar{\chi
}\left(  -2\right)  G\left(  \bar{\chi}\right)  B_{1}\left(  z,1-p:\chi
\right)  -\frac{\left(  2\pi i\right)  ^{p}}{2\left(  p!\right)  }g_{1}\left(
d,-c-dk,z,p,\chi\right)  \right)  . \label{B}%
\end{align}
Taking $a=-k,b=k^{2}-1,c=1$ and $d=-k$ and writing $\bar{\chi}$ in place of
$\chi$ in (\ref{B1b}) lead to%
\begin{align}
&  \left(  z-k\right)  ^{p-1}2^{p}\chi\left(  2\right)  G\left(  \chi\right)
B_{1}\left(  V_{1}\left(  z\right)  ,1-p:\bar{\chi}\right) \nonumber\\
&  =\chi\left(  \frac{k^{2}-1}{2}\right)  \bar{\chi}\left(  2\right)
2^{p}\chi\left(  2\right)  G\left(  \bar{\chi}\right)  B_{1}\left(
z,1-p:\chi\right) \nonumber\\
&  \quad-\chi\left(  \frac{k^{2}-1}{2}\right)  \chi\left(  -1\right)
2^{p}\frac{\left(  2\pi i\right)  ^{p}}{\left(  p+1\right)  !}{\sum
\limits_{m=1}^{p}}\binom{p+1}{m}k^{m-p}\left(  -\left(  z-k\right)  \right)
^{m-1}\nonumber\\
&  \qquad\times\frac{m}{2^{m}}{\sum\limits_{n=1}^{k}}\bar{\chi}\left(
n\right)  \mathfrak{B}_{p+1-m,\bar{\chi}}\left(  \frac{-kn}{2}\right)
\mathcal{E}_{m-1}\left(  \frac{n}{k}\right)  . \label{C}%
\end{align}
Thus, consider (\ref{B}) and\ (\ref{C}) with multiplying both sides of
(\ref{A}) by\ $\left(  z-k\right)  ^{p-1}$ to obtain%
\begin{align}
&  g_{1}\left(  d,-c-dk,z,p,\chi\right)  -\chi\left(  -1\right)  \left(
z-k\right)  ^{p-1}g_{1}\left(  c,d+ck,V_{1}\left(  z\right)  ,p,\bar{\chi
}\right) \nonumber\\
&  \quad=\bar{\chi}\left(  -2\right)  \frac{2^{p+1}}{p+1}{\sum\limits_{m=1}%
^{p}}\binom{p+1}{m}k^{m-p}\left(  -\left(  z-k\right)  \right)  ^{m-1}%
\nonumber\\
&  \qquad\times\frac{m}{2^{m}}{\sum\limits_{n=1}^{k}}\bar{\chi}\left(
n\right)  \mathfrak{B}_{p+1-m,\bar{\chi}}\left(  \frac{-kn}{2}\right)
\mathcal{E}_{m-1}\left(  \frac{n}{k}\right)  . \label{53}%
\end{align}
Now, let us concern the sum over $n$ in (\ref{53}). Using (\ref{6}) for $r=2$
yields
\begin{align*}
&  {\sum\limits_{n=1}^{k}}\bar{\chi}\left(  n\right)  \mathfrak{B}%
_{p+1-m,\bar{\chi}}\left(  \frac{-kn}{2}\right)  \mathcal{E}_{m-1}\left(
\frac{n}{k}\right) \\
&  =\mathfrak{B}_{p+1-m,\bar{\chi}}\left(  0\right)  \sum_{n}\bar{\chi}\left(
2n\right)  \mathcal{E}_{m-1}\left(  \frac{2n}{k}\right)  +\mathfrak{B}%
_{p+1-m,\bar{\chi}}\left(  \frac{k}{2}\right)  \sum_{n}\bar{\chi}\left(
2n-1\right)  \mathcal{E}_{m-1}\left(  \frac{2n-1}{k}\right) \\
&  =\left\{  \mathfrak{B}_{p+1-m,\bar{\chi}}\left(  0\right)  +\mathfrak{B}%
_{p+1-m,\bar{\chi}}\left(  \frac{k}{2}\right)  \right\}  \sum_{n}\bar{\chi
}\left(  2n\right)  \mathcal{E}_{m-1}\left(  \frac{2n}{k}\right) \\
&  \quad-\mathfrak{B}_{p+1-m,\bar{\chi}}\left(  \frac{k}{2}\right)  \sum
_{n=0}^{k-1}\left(  -1\right)  ^{n}\bar{\chi}\left(  n\right)  \mathcal{E}%
_{m-1}\left(  \frac{n}{k}\right) \\
&  =2^{m-p}\chi\left(  2\right)  \mathfrak{B}_{p+1-m,\bar{\chi}}\left(
0\right)  \sum_{n=0}^{\left(  k-1\right)  /2}\bar{\chi}\left(  2n\right)
\mathcal{E}_{m-1}\left(  \frac{2n}{k}\right) \\
&  \quad-\mathfrak{B}_{p+1-m,\bar{\chi}}\left(  \frac{k}{2}\right)
k^{1-m}\mathcal{E}_{m-1,\chi}\left(  0\right)  .
\end{align*}
It follows from (\ref{e1}) that
\begin{align}
&  \sum_{\mu=0}^{k-1}\chi\left(  2\mu\right)  \mathcal{E}_{m-1}\left(
\frac{2\mu}{k}\right) \nonumber\\
&  =\sum_{\mu=0}^{\left(  k-1\right)  /2}\chi\left(  2\mu\right)
\mathcal{E}_{m-1}\left(  \frac{2\mu}{k}\right)  +\sum_{\mu=\left(  k+1\right)
/2}^{k-1}\chi\left(  2\mu\right)  \mathcal{E}_{m-1}\left(  \frac{2\mu}%
{k}\right) \nonumber\\
&  =\sum_{\mu=0}^{\left(  k-1\right)  /2}\chi\left(  2\mu\right)
\mathcal{E}_{m-1}\left(  \frac{2\mu}{k}\right)  +\sum_{\mu=1}^{\left(
k-1\right)  /2}\chi\left(  -2\mu\right)  \mathcal{E}_{m-1}\left(  -\frac{2\mu
}{k}\right) \nonumber\\
&  =\left(  1+\left(  -1\right)  ^{m}\chi\left(  -1\right)  \right)  \sum
_{\mu=1}^{\left(  k-1\right)  /2}\chi\left(  2\mu\right)  \mathcal{E}%
_{m-1}\left(  \frac{2\mu}{k}\right)  .\nonumber
\end{align}
So, using (\ref{b}) gives
\begin{align*}
&  \mathfrak{B}_{p+1-m,\bar{\chi}}\left(  0\right)  \sum_{n=0}^{k-1}\bar{\chi
}\left(  2n\right)  \mathcal{E}_{m-1}\left(  \frac{2n}{k}\right) \\
&  =\left\{  \mathfrak{B}_{p+1-m,\bar{\chi}}\left(  0\right)  +\left(
-1\right)  ^{m}\chi\left(  -1\right)  \mathfrak{B}_{p+1-m,\bar{\chi}}\left(
0\right)  \right\}  \sum_{n=1}^{\left(  k-1\right)  /2}\bar{\chi}\left(
2n\right)  \mathcal{E}_{m-1}\left(  \frac{2n}{k}\right) \\
&  \ =2\mathfrak{B}_{p+1-m,\bar{\chi}}\left(  0\right)  \sum_{n=1}^{\left(
k-1\right)  /2}\bar{\chi}\left(  2n\right)  \mathcal{E}_{m-1}\left(  \frac
{2n}{k}\right)  .
\end{align*}
Here, using (\ref{12}) and taking $x=k$ in (\ref{be}) give rise to%
\begin{align}
&  {\sum\limits_{n=1}^{k}}\bar{\chi}\left(  n\right)  \mathfrak{B}%
_{p+1-m,\bar{\chi}}\left(  \frac{-kn}{2}\right)  \mathcal{E}_{m-1}\left(
\frac{n}{k}\right) \nonumber\\
&  \ =\frac{\bar{\chi}\left(  2\right)  }{2^{p+1-m}}\left\{  \mathfrak{B}%
_{p+1-m,\bar{\chi}}\left(  0\right)  -2^{p+1-m}\chi\left(  2\right)
\mathfrak{B}_{p+1-m,\bar{\chi}}\left(  \frac{k}{2}\right)  \right\}
k^{1-m}\mathcal{E}_{m-1,\chi}\left(  0\right) \nonumber\\
&  \ =k^{1-m}2^{m-p-2}\bar{\chi}\left(  2\right)  \left(  p+1-m\right)
\mathcal{E}_{p-m,\bar{\chi}}\left(  k\right)  \mathcal{E}_{m-1,\chi}\left(
0\right)  . \label{54}%
\end{align}
Gathering (\ref{53}), (\ref{54}) and (\ref{e}) completes the proof.
\end{proof}

Theorem \ref{teoRP2} can be simplified according to special values of $z.$
Firstly, let us consider the case $z=\dfrac{c}{d}+k$. Then,
\begin{equation}
g_{1}\left(  d,-c-dk,\dfrac{c}{d}+k,p,\chi\right)  =\frac{p}{k^{p-1}}%
\sum_{n=1}^{dk}\left(  -1\right)  ^{n}\chi\left(  n\right)  \mathcal{E}%
_{p-1,\chi}\left(  \frac{-cn}{d}\right)  \mathfrak{B}_{1}\left(  \frac{n}%
{dk}\right)  \label{4-1}%
\end{equation}
and
\begin{align}
&  \left(  \dfrac{c}{d}\right)  ^{p-1}g_{1}\left(  c,d+ck,V_{1}\left(
\dfrac{c}{d}+k\right)  ,p,\bar{\chi}\right) \nonumber\\
&  \quad=\left(  \frac{c}{kd}\right)  ^{p-1}p\sum_{n=1}^{ck}\left(  -1\right)
^{n}\bar{\chi}\left(  n\right)  \mathcal{E}_{p-1,\bar{\chi}}\left(  \frac
{dn}{c}\right)  \mathfrak{B}_{1}\left(  \frac{n}{ck}\right)  . \label{4-2}%
\end{align}

Since $\mathcal{E}_{0}\left(  x\right)  =\left(  -1\right)  ^{\left[
x\right]  }E_{0}\left(  \left\{  x\right\}  \right)  =\left(  -1\right)
^{\left[  x\right]  }$ and
\[
s_{5}(d,c)=\sum\limits_{n=1}^{c-1}\left(  -1\right)  ^{n+\left[  dn/c\right]
}\mathfrak{B}_{1}\left(  \frac{n}{c}\right)  =\sum\limits_{n=1}^{c-1}\left(
-1\right)  ^{n}\mathcal{E}_{0}\left(  \frac{dn}{c}\right)  \mathfrak{B}%
_{1}\left(  \frac{n}{c}\right)  ,
\]
it is convenient to make the following definition.

\begin{definition}
The character Hardy--Berndt sum $s_{5,p}\left(  d,c:\chi\right)  $ is defined
for $c>0$ by%
\[
s_{5,p}\left(  d,c:\chi\right)  =\sum_{n=1}^{ck}\left(  -1\right)  ^{n}%
\chi\left(  n\right)  \mathcal{E}_{p-1,\chi}\left(  \frac{dn}{c}\right)
\mathfrak{B}_{1}\left(  \frac{n}{ck}\right)  .
\]

\end{definition}

Observing that
\[
s_{5,p}\left(  -c,d:\chi\right)  =-\chi\left(  -1\right)  s_{5,p}\left(
c,d:\chi\right)  ,
\]
by (\ref{e}), and using (\ref{4-1}) and (\ref{4-2}) in Theorem \ref{teoRP2} we
have proved the following reciprocity formula.

\begin{theorem}
\label{rps5}Let $p\geq1$ be odd and $d$ and $c$ be odd coprime integers. If
$c$ or $d\equiv0$ $\left(  \operatorname{mod}k\right)  ,$ then
\begin{align*}
&  cd^{p}s_{5,p}\left(  c,d:\chi\right)  +dc^{p}s_{5,p}\left(  d,c:\bar{\chi
}\right) \\
&  =-\frac{\bar{\chi}\left(  -4\right)  }{2}{\sum\limits_{m=0}^{p-1}}%
\binom{p-1}{m}c^{m+1}d^{p-m}\mathcal{E}_{p-1-m,\bar{\chi}}\left(  0\right)
\mathcal{E}_{m,\chi}\left(  0\right)  .
\end{align*}

\end{theorem}

In particular, we have the character analogue of (\ref{9}) as
\[
s_{5}\left(  c,d:\chi\right)  +s_{5}\left(  d,c:\bar{\chi}\right)
=-\frac{\bar{\chi}\left(  -4\right)  }{2}\mathcal{E}_{0,\bar{\chi}}\left(
0\right)  \mathcal{E}_{0,\chi}\left(  0\right)  ,
\]
where $s_{5}\left(  c,d:\chi\right)  =s_{5,1}\left(  c,d:\chi\right)  $.

Now we let $z=k$ in Theorem \ref{teoRP2}. Then%
\begin{align*}
&  g_{1}\left(  d,-c-dk,z,p,\chi\right)  |_{z=k}\\
&  =\sum_{m=1}^{p}\binom{p}{m}k^{m-p}c^{m-1}\sum_{n=1}^{dk}\left(  -1\right)
^{n}\chi\left(  n\right)  \mathcal{E}_{p-m,\chi}\left(  \frac{-cn}{d}\right)
\mathfrak{B}_{m}\left(  \frac{n}{dk}\right)
\end{align*}
and%
\begin{align*}
&  \left(  z-k\right)  ^{p-1}g_{1}\left(  c,d+ck,V_{1}\left(  z\right)
,p,\bar{\chi}\right)  |_{z=k}\\
&  =\sum_{m=1}^{p}\binom{p}{m}k^{m-p}\left(  -\left(  c\left(  -kz+k^{2}%
-1\right)  +\left(  d+ck\right)  \left(  z-k\right)  \right)  \right)
^{m-1}\left(  z-k\right)  ^{p-m}|_{z=k}\\
&  \qquad\times\sum_{n=1}^{ck}\left(  -1\right)  ^{n}\bar{\chi}\left(
n\right)  \mathcal{E}_{p-m,\bar{\chi}}\left(  \frac{dn}{c}\right)
\mathfrak{B}_{m}\left(  \frac{n}{ck}\right) \\
&  =c^{p-1}\sum_{n=1}^{ck}\left(  -1\right)  ^{n}\bar{\chi}\left(  n\right)
\mathcal{E}_{0,\bar{\chi}}\left(  \frac{dn}{c}\right)  \mathfrak{B}_{p}\left(
\frac{n}{ck}\right)  .
\end{align*}
If we define
\[
s_{5,p+1-m,m}\left(  c,d:\chi\right)  =\sum_{n=1}^{dk}\left(  -1\right)
^{n}\chi\left(  n\right)  \mathcal{E}_{p-m,\chi}\left(  \frac{cn}{d}\right)
\mathfrak{B}_{m}\left(  \frac{n}{dk}\right)
\]
and use (\ref{e}), then we see that%
\begin{align*}
&  \sum_{m=1}^{p}\binom{p}{m}\left(  -kc\right)  ^{m-1}s_{5,p+1-m,m}\left(
c,d:\chi\right) \\
&  \quad=-\left(  kc\right)  ^{p-1}s_{5,1,p}\left(  d,c:\bar{\chi}\right)
-\bar{\chi}\left(  -4\right)  \frac{p}{2}\mathcal{E}_{p-1,\bar{\chi}}\left(
0\right)  \mathcal{E}_{0,\chi}\left(  0\right)  .
\end{align*}

Conditions $d$\ even and $c$\ even in Theorem \ref{teoS1}\ and Theorem
\ref{teoS2} do not allow to present reciprocity theorems in the sense of
Theorem \ref{teoRP2} for the functions $g_{1}(d,c,z,p,\chi)$ and
$g_{2}(d,c,z,p,\chi),$ respectively. However, the following relation is valid
for these functions.

\begin{theorem}
\label{teoRP1}Let $d$ be even. If $d$ or $c\equiv0$ $\left(
\operatorname{mod}k\right)  ,$ then
\begin{align}
&  \frac{p+1}{2}z^{p-1}g_{1}\left(  c,d,-\frac{1}{z},p,\bar{\chi}\right)
+g_{2}\left(  d,-c,z,p,\chi\right) \nonumber\\
&  =-\frac{\chi\left(  -1\right)  }{k^{p-1}}\sum\limits_{m=1}^{p}\binom
{p+1}{m}\frac{m}{2}\left(  -z\right)  ^{m-1}\mathfrak{B}_{p+1-m,\chi}\left(
0\right)  \mathcal{E}_{m-1,\bar{\chi}}\left(  0\right)  , \label{rsp1}%
\end{align}
where the functions $g_{1}(d,c,z,p,\chi)$ and $g_{2}(d,c,z,p,\chi)$ are given
by (\ref{g1}) and (\ref{g2}), respectively.
\end{theorem}

\begin{proof}
For even $d$, consider $T\left(  z\right)  =\left(  az+b\right)  /\left(
cz+d\right)  $ and $T^{\ast}\left(  z\right)  =\left(  bz-a\right)  /\left(
dz-c\right)  =T\left(  -1/z\right)  $ and $a\equiv d\equiv0\left(
\operatorname{mod}k\right)  .$ Then, (\ref{rsp1}) follows by applying
$T^{\ast}\left(  z\right)  $ in (\ref{bc2}) and replacing $z$ by $-1/z$ in
(\ref{ad1}), and then replacing $T\left(  z\right)  $ by $-1/z$ in (\ref{B1a}).
\end{proof}

To simplify Theorem \ref{teoRP1} we first consider $z=c/d$. Then,%
\begin{equation}
g_{1}\left(  c,d,-\frac{d}{c},p,\bar{\chi}\right)  =\frac{p}{k^{p-1}}%
\sum_{n=1}^{ck}\bar{\chi}\left(  n\right)  \mathcal{E}_{p-1,\bar{\chi}}\left(
\frac{dn}{c}\right)  \mathfrak{B}_{1}\left(  \frac{n}{ck}\right)  \label{5-1}%
\end{equation}
and%
\begin{equation}
g_{2}\left(  d,-c,\frac{c}{d},p,\bar{\chi}\right)  =\frac{p+1}{k^{p-1}}%
\sum\limits_{n=1}^{dk}\left(  -1\right)  ^{n}\chi\left(  n\right)
\mathfrak{B}_{p,\chi}\left(  \frac{-cn}{d}\right)  \mathfrak{B}_{1}\left(
\frac{n}{dk}\right)  . \label{5-2}%
\end{equation}

\begin{definition}
The character Hardy--Berndt sums $s_{1,p}\left(  d,c,\chi\right)  $ and
$s_{2,p}\left(  d,c,\chi\right)  $ are defined for $c>0$ by%
\begin{align*}
s_{1,p}\left(  d,c:\chi\right)   &  =\sum_{n=1}^{ck}\chi\left(  n\right)
\mathcal{E}_{p-1,\chi}\left(  \frac{dn}{c}\right)  \mathfrak{B}_{1}\left(
\frac{n}{ck}\right)  ,\\
s_{2,p}\left(  d,c:\chi\right)   &  =\sum\limits_{n=1}^{ck}\left(  -1\right)
^{n}\chi\left(  n\right)  \mathfrak{B}_{p,\chi}\left(  \frac{dn}{c}\right)
\mathfrak{B}_{1}\left(  \frac{n}{ck}\right)  .
\end{align*}

\end{definition}

Using (\ref{5-1}) and (\ref{5-2}) in Theorem \ref{teoRP1} we have proved the
following reciprocity formula for $s_{1,p}\left(  d,c,\chi\right)  $ and
$s_{2,p}\left(  d,c,\chi\right)  $.

\begin{theorem}
\label{rps1s2}Let $p\geq1$ be odd, $\left(  d,c\right)  =1$ and $d$ be even.
If $d$ or $c\equiv0\left(  \operatorname{mod}k\right)  ,$ then%
\begin{align*}
&  pdc^{p}s_{1,p}\left(  d,c:\bar{\chi}\right)  -\chi\left(  -1\right)
2cd^{p}s_{2,p}\left(  c,d:\chi\right) \\
&  \quad=\chi\left(  -1\right)  \sum\limits_{m=1}^{p}\left(  -1\right)
^{m}\binom{p}{m-1}c^{m}d^{p+1-m}\mathfrak{B}_{p+1-m,\chi}\left(  0\right)
\mathcal{E}_{m-1,\bar{\chi}}\left(  0\right)  .
\end{align*}
In particular,
\begin{equation}
s_{1}\left(  d,c:\bar{\chi}\right)  -2\chi\left(  -1\right)  s_{2}\left(
c,d:\chi\right)  =-\chi\left(  -1\right)  \mathfrak{B}_{1,\chi}\left(
0\right)  \mathcal{E}_{0,\bar{\chi}}\left(  0\right)  .\nonumber
\end{equation}

\end{theorem}

\begin{remark}
The sum $s_{2}\left(  c,d:\chi\right)  =s_{2,1}\left(  c,d:\chi\right)  $ is
first presented by Meyer\ \cite[Definition 6]{m3} as $s_{5}^{\ast}\left(
d,c:\chi\right)  $.
\end{remark}

Now letting $z=0$ in Theorem \ref{teoRP1} we find that
\begin{align*}
&  \left.  \frac{p+1}{2}z^{p-1}g_{1}\left(  c,d,-\frac{1}{z},p,\chi\right)
\right\vert _{z=0}\\
&  =\frac{p+1}{2}\left.  \sum_{m=1}^{p}\binom{p}{m}k^{m-p}\left(  -\left(
dz-c\right)  \right)  ^{m-1}z^{p-m}s_{1,p+1-m,m}\left(  d,c:\chi\right)
\right\vert _{z=0}\\
&  =\frac{p+1}{2}c^{p-1}s_{1,1,p}\left(  d,c:\chi\right)
\end{align*}
and%
\[
g_{2}\left(  d,-c,0,p,\bar{\chi}\right)  =\sum_{m=1}^{p}\binom{p+1}{m}%
k^{m-p}c^{m-1}s_{2,p+1-m,m}\left(  -c,d:\bar{\chi}\right)
\]
where
\begin{align*}
s_{1,p+1-m,m}\left(  d,c:\chi\right)   &  =\sum_{n=1}^{ck}\chi\left(
n\right)  \mathcal{E}_{p-m,\chi}\left(  \frac{dn}{c}\right)  \mathfrak{B}%
_{m}\left(  \frac{n}{ck}\right)  ,\\
s_{2,p+1-m,m}\left(  d,c:\chi\right)   &  =\sum\limits_{n=1}^{ck}\left(
-1\right)  ^{n}\chi\left(  n\right)  \mathfrak{B}_{p+1-m,\chi}\left(
\frac{dn}{c}\right)  \mathfrak{B}_{m}\left(  \frac{n}{ck}\right)  .
\end{align*}
Using (\ref{b}), Theorem \ref{teoRP1} reduces to
\begin{align*}
&  \sum_{m=1}^{p}\binom{p+1}{m}\left(  -ck\right)  ^{m-1}s_{2,p+1-m,m}\left(
c,d:\bar{\chi}\right) \\
&  =\frac{p+1}{2}\left(  \chi\left(  -1\right)  \left(  ck\right)
^{p-1}s_{1,1,p}\left(  d,c:\chi\right)  +\mathfrak{B}_{p,\chi}\left(
0\right)  \mathcal{E}_{0,\bar{\chi}}\left(  0\right)  \right)  .
\end{align*}

The following lemma shows that reciprocity formulas given by Theorems
\ref{rps5} and \ref{rps1s2} are still valid for $\gcd(d,c)=q.$

\begin{lemma}
Let $q\in\mathbb{N}$, $p\geq1$, $\left(  d,c\right)  =1$\ and $c>0.$ If $p$ is
odd and $d$ is even,%
\[
s_{1,p}\left(  qd,qc:\chi\right)  =s_{1,p}\left(  d,c:\chi\right)  ,
\]
if $p$ is odd and $c$ is even,%
\[
s_{2,p}\left(  qd,qc:\chi\right)  =s_{2,p}\left(  d,c:\chi\right)  ,
\]
if $p$ is odd and $\left(  d+c\right)  $\ is even,
\[
s_{5,p}\left(  qd,qc:\chi\right)  =s_{5,p}\left(  d,c:\chi\right)  .
\]
Furthermore, $s_{1,p}\left(  d,c:\chi\right)  =0$ if $\left(  d+p\right)  $ is
even$,$ $s_{2,p}\left(  d,c:\chi\right)  =0$ if $\left(  c+p\right)  $ is even
and $s_{5,p}\left(  d,c:\chi\right)  =0$ if $\left(  d+c+p\right)  $ is even.
\end{lemma}

\begin{proof}
Let $p$ be odd and $d$ be even. Then, setting $\mu=n+mck,$ $1\leq n\leq ck,$
$0\leq m\leq q-1$ and using (\ref{1}) and (\ref{e}) yield
\begin{align*}
s_{1,p}\left(  qd,qc:\chi\right)   &  =\sum_{\mu=1}^{qck}\chi\left(
\mu\right)  \mathcal{E}_{p-1,\chi}\left(  \frac{d\mu}{c}\right)
\mathfrak{B}_{1}\left(  \frac{\mu}{qck}\right) \\
&  =\sum_{n=1}^{ck}\chi\left(  n\right)  \mathcal{E}_{p-1,\chi}\left(
\frac{dn}{c}\right)  \sum_{m=0}^{q-1}\left(  -1\right)  ^{dm}\mathfrak{B}%
_{1}\left(  \frac{n}{qck}+\frac{m}{q}\right) \\
&  =s_{1,p}\left(  d,c:\chi\right)  .
\end{align*}
On the other hand, using (\ref{e}),%
\begin{align*}
s_{1,p}\left(  d,c:\chi\right)   &  =\sum_{n=1}^{ck}\chi\left(  n\right)
\mathcal{E}_{p-1,\chi}\left(  \frac{dn}{c}\right)  \mathfrak{B}_{1}\left(
\frac{n}{ck}\right) \\
&  =\sum_{n=1}^{ck}\chi\left(  -n\right)  \mathcal{E}_{p-1,\chi}\left(
dk-\frac{dn}{c}\right)  \mathfrak{B}_{1}\left(  1-\frac{n}{ck}\right) \\
&  =\left(  -1\right)  ^{d+p+1}s_{1,p}\left(  d,c:\chi\right)
\end{align*}
which leads to $s_{1,p}\left(  d,c:\chi\right)  =0$ for even $d+p$.

Other statements can be shown in a similar way.
\end{proof}

\section{Some series relations}

In this final section, we deal with (\ref{ad1}) and (\ref{ad2}) for special
values of $Tz$ to present series relations, motivated by \cite{b4} (see also
\cite{dc2,m3}).

Summing over $m$ we see that%
\begin{equation}
G\left(  \bar{\chi}\right)  A_{1}\left(  z,1-p:\chi\right)  =-\sum
\limits_{j=1}^{k-1}\bar{\chi}\left(  j\right)  \sum\limits_{n=1}^{\infty}%
\frac{\chi\left(  n\right)  }{n^{p}\left(  e^{-2\pi i\left(  j+nz\right)
/k}+1\right)  } \label{19}%
\end{equation}
and%
\begin{equation}
G\left(  \chi\right)  B\left(  z,1-p:\bar{\chi}\right)  =\sum\limits_{j=1}%
^{k-1}\chi\left(  j\right)  \sum\limits_{n=0}^{\infty}\frac{\bar{\chi}\left(
2n+1\right)  }{\left(  2n+1\right)  ^{p}\left(  e^{-\pi i\left(  2j+\left(
2n+1\right)  z\right)  /k}-1\right)  }. \label{20}%
\end{equation}

\begin{theorem}
\label{seri1}Let $p\geq1$ be odd and $\alpha\beta=\left(  \pi/k\right)  ^{2}$
with $\alpha,\beta>0$. Then,%
\begin{align*}
&  \left(  -\beta\right)  ^{\left(  p-1\right)  /2}\sum\limits_{j=1}^{k-1}%
\bar{\chi}\left(  j\right)  \sum\limits_{n=1}^{\infty}\frac{\chi\left(
n\right)  }{n^{p}\left(  e^{2n\alpha-2\pi ij/k}+1\right)  }\\
&  +2^{p}\alpha^{\left(  p-1\right)  /2}\chi\left(  -2\right)  \sum
\limits_{j=1}^{k-1}\chi\left(  j\right)  \sum\limits_{n=0}^{\infty}\frac
{\bar{\chi}\left(  2n+1\right)  }{\left(  2n+1\right)  ^{p}\left(  e^{\left(
2n+1\right)  \beta-2\pi ij/k}-1\right)  }\\
&  =2^{p-2}\frac{k}{p!}\sum_{m=1}^{p}\binom{p}{m}\left(  i\right)
^{p+1-m}\mathcal{E}_{p-m,\bar{\chi}}\left(  0\right)  \mathfrak{B}_{m,\chi
}\left(  0\right)  \alpha^{p-m/2}\beta^{\left(  p-1+m\right)  /2}.
\end{align*}

\end{theorem}

\begin{proof}
We put $a=d=0,$ $b=-1$ and $c=1$ in (\ref{ad1}) to obtain
\begin{align*}
z^{p-1}G\left(  \bar{\chi}\right)  A_{1}\left(  -\frac{1}{z},1-p:\chi\right)
&  =2^{p}\chi\left(  -2\right)  G\left(  \chi\right)  B\left(  z,1-p:\bar
{\chi}\right) \\
&  -\frac{k}{4}\frac{\left(  2\pi i/k\right)  ^{p}}{p!}\sum_{m=1}^{p}\binom
{p}{m}\mathcal{E}_{p-m,\bar{\chi}}\left(  0\right)  \mathfrak{B}_{m,\chi
}\left(  0\right)  \left(  -z\right)  ^{m-1}.
\end{align*}
Then, the proof follows by setting $z=\pi i/k\alpha$ and using (\ref{19}),
(\ref{20}) and that $\alpha\beta=\left(  \pi/k\right)  ^{2},$ and then
multiplying both sides by $\alpha^{\left(  p-1\right)  /2}.$
\end{proof}

\begin{corollary}
Let $p\geq1$ be odd and let $\chi$ be the primitive character of modulus $3$
defined by%
\begin{equation}
\chi\left(  n\right)  =\left\{
\begin{array}
[c]{cc}%
1, & n\equiv1\left(  \operatorname{mod}3\right)  ,\\
-1, & n\equiv2\left(  \operatorname{mod}3\right)  ,\\
0, & n\equiv0\left(  \operatorname{mod}3\right)  .
\end{array}
\right.  \label{16}%
\end{equation}
Then,
\begin{align}
&  \sum\limits_{n=1}^{\infty}\left(  -1\right)  ^{n\left(  p+1\right)
/2}\frac{\chi\left(  n\right)  }{n^{p}\left(  2\cosh\left(  n\pi/3\right)
-\left(  -1\right)  ^{n}\right)  }\nonumber\\
&  =\frac{\left(  -1\right)  ^{\left(  p+1\right)  /2}}{4\sqrt{3}}\frac{3}%
{p!}\left(  \frac{\pi}{3}\right)  ^{p}\sum_{m=1}^{p}\binom{p}{m}\left(
i\right)  ^{p-m}\mathcal{E}_{p-m,\chi}\left(  0\right)  \mathfrak{B}_{m,\chi
}\left(  0\right)  . \label{21}%
\end{align}
In particular,
\[
\sum\limits_{n=1}^{\infty}\frac{\left(  -1\right)  ^{n}\chi\left(  n\right)
}{n\left(  2\cosh\left(  n\pi/3\right)  -\left(  -1\right)  ^{n}\right)
}=-\frac{\pi}{4\sqrt{3}}\mathcal{E}_{0,\chi}\left(  0\right)  \mathfrak{B}%
_{1,\chi}\left(  0\right)  \mathcal{=-}\frac{\pi}{6\sqrt{3}}%
\]
and\textbf{\ }%
\[
\sum\limits_{n=1}^{\infty}\frac{\chi\left(  n\right)  }{n^{3}\left(
2\cosh\left(  n\pi/3\right)  -\left(  -1\right)  ^{n}\right)  }=\frac{\left(
\pi/3\right)  ^{3}}{8\sqrt{3}}\left(  \mathcal{E}_{0,\chi}\left(  0\right)
\mathfrak{B}_{3,\chi}\left(  0\right)  -3\mathcal{E}_{2,\chi}\left(  0\right)
\mathfrak{B}_{1,\chi}\left(  0\right)  \right)
\]

\end{corollary}

\begin{proof}
Setting $\alpha=\beta=\pi/3$ in Theorem \ref{seri1} we have
\begin{align*}
&  \left(  -1\right)  ^{\left(  p-1\right)  /2}\sum\limits_{n=1}^{\infty}%
\frac{\chi\left(  n\right)  }{n^{p}}\left(  \frac{1}{e^{2n\alpha-2\pi i/3}%
+1}-\frac{1}{e^{2n\alpha-4\pi i/3}+1}\right) \\
&  \quad+2^{p}\sum\limits_{n=0}^{\infty}\frac{\chi\left(  2n+1\right)
}{\left(  2n+1\right)  ^{p}}\left(  \frac{1}{e^{\left(  2n+1\right)
\alpha-2\pi i/3}-1}-\frac{1}{e^{\left(  2n+1\right)  \alpha-4\pi i/3}%
-1}\right) \\
&  =2^{p-2}\frac{3}{p!}\left(  \frac{\pi}{3}\right)  ^{p}\sum_{m=1}^{p}%
\binom{p}{m}\left(  i\right)  ^{p+1-m}\mathcal{E}_{p-m,\chi}\left(  0\right)
\mathfrak{B}_{m,\chi}\left(  0\right)  .
\end{align*}
Some simplification gives
\begin{align*}
&  \sum\limits_{n=1}^{\infty}\frac{\left(  -1\right)  ^{\left(  p+1\right)
/2}\chi\left(  2n\right)  }{\left(  2n\right)  ^{p}\left(  2\cosh
2n\alpha-1\right)  }+\sum\limits_{n=0}^{\infty}\frac{\chi\left(  2n+1\right)
}{\left(  2n+1\right)  ^{p}\left(  2\cosh\left(  2n+1\right)  \alpha+1\right)
}\\
&  =\frac{1}{4\sqrt{3}}\frac{3}{p!}\left(  \frac{\pi}{3}\right)  ^{p}%
\sum_{m=1}^{p}\binom{p}{m}\left(  i\right)  ^{p-m}\mathcal{E}_{p-m,\bar{\chi}%
}\left(  0\right)  \mathfrak{B}_{m,\chi}\left(  0\right)  ,
\end{align*}
which is equivalent to (\ref{21}).
\end{proof}

Observe that for $\alpha=\beta=\pi/k$ and real-valued primitive character
$\chi,$ Theorem \ref{seri1} can be composed as%
\begin{align*}
&  \chi\left(  2\right)  \sum\limits_{j=1}^{k-1}\chi\left(  j\right)
\sum\limits_{n=1}^{\infty}\frac{\left(  -1\right)  ^{n\delta}\chi\left(
n\right)  }{n^{p}\left(  e^{n\alpha-2\pi ij/k}+\left(  -1\right)  ^{n}\right)
}\\
&  =\frac{\left(  -1\right)  ^{\left(  p-1\right)  /2}}{4}\left(  \frac{\pi
}{k}\right)  ^{p}\frac{k}{p!}\sum_{m=1}^{p}\binom{p}{m}\left(  i\right)
^{p+1-m}\mathcal{E}_{p-m,\bar{\chi}}\left(  0\right)  \mathfrak{B}_{m,\chi
}\left(  0\right)  ,
\end{align*}
where $\delta=\left\{
\begin{array}
[c]{cc}%
\left(  p-1\right)  /2, & \text{if }\chi\left(  -1\right)  =1,\\
\left(  p+1\right)  /2, & \text{if }\chi\left(  -1\right)  =-1.
\end{array}
\right.  $

\begin{theorem}
Let $\alpha\beta=\left(  \pi/3\right)  ^{2}$ with $\alpha,\beta>0$ and let
$\chi$ be the primitive character of modulus $3$ given by (\ref{16}). Then
\[
\sum\limits_{n=1}^{\infty}\frac{\left(  -1\right)  ^{n}\chi\left(  n\right)
}{n\left(  2\cosh2n\alpha-\left(  -1\right)  ^{n}\right)  }+\sum
\limits_{n=1}^{\infty}\frac{\left(  -1\right)  ^{n}\chi\left(  n\right)
}{n\left(  2\cosh2n\beta-\left(  -1\right)  ^{n}\right)  }=-\frac{\pi}%
{3\sqrt{3}}.
\]
In particular,
\[
\sum\limits_{n=1}^{\infty}\frac{\left(  -1\right)  ^{n}\chi\left(  n\right)
}{n\left(  2\cosh\frac{2n\pi}{3}-\left(  -1\right)  ^{n}\right)  }=-\frac{\pi
}{6\sqrt{3}}.
\]

\end{theorem}

\begin{proof}
We apply (\ref{ad2}) with $a=d=3,$ $b=4$ and $c=2.$ Setting $2z+3=\pi
i/3\alpha$ we have $Tz=3\left(  1-\alpha/\pi i\right)  /2$ and $z=-3\left(
1+\beta/\pi i\right)  /2,$ where $\alpha\beta=\left(  \pi/3\right)  ^{2}.$
Straightforward calculation gives
\begin{align*}
&  \sum\limits_{n=1}^{\infty}\frac{\chi\left(  n\right)  }{n\left(
2\cosh\left(  2n\alpha-n\pi i\right)  -1\right)  }+\sum\limits_{n=1}^{\infty
}\frac{\chi\left(  n\right)  }{n\left(  2\cosh\left(  2n\beta+n\pi i\right)
-1\right)  }\\
&  =-\frac{\pi}{\sqrt{3}}\sum\limits_{j=1}^{5}\left(  -1\right)  ^{j}%
\chi\left(  j\right)  \mathfrak{B}_{1,\chi}\left(  \frac{3j}{2}\right)
\mathfrak{B}_{1}\left(  \frac{j}{6}\right)  .
\end{align*}
Using (\ref{3}) and the fact $\mathfrak{B}_{1}\left(  x+1\right)
=\mathfrak{B}_{1}\left(  x\right)  =x-1/2$ when $0<x<1,$ we find that
\[
\sum\limits_{j=1}^{5}\left(  -1\right)  ^{j}\chi\left(  j\right)
\mathfrak{B}_{1,\chi}\left(  \frac{3j}{2}\right)  \mathfrak{B}_{1}\left(
\frac{j}{6}\right)  =\frac{1}{3}%
\]
which completes the proof.
\end{proof}

\end{document}